\documentclass[12pt]{article}
\usepackage{amsmath, amsthm, amssymb}
\usepackage{graphicx} % Required for inserting images
\usepackage{tikz}

\usepackage{booktabs} % See the package documentation for guidelines on formal tables: https://ctan.org/pkg/booktabs
\usepackage{verbatim} % Used to typeset, for example, code snippets or pseudo-code for algorithms.
\usepackage{dsfont} % Extra fontset for helpful mathematics symbols, e.g. \mathds{1}
\usepackage{etoolbox} % Used to allow boolean variables for use in the title page
\usepackage{pgf, subcaption, tabularx}
\pgfdeclarelayer{edgelayer}
\pgfdeclarelayer{nodelayer}
\pgfsetlayers{edgelayer,nodelayer,main}

\usepackage{enumitem}
\usepackage{mdframed}
\usepackage{url}
\usepackage{adjustbox}
\usepackage{parskip}
\usepackage{colortbl,hhline}

\tikzset{
	white-vertex/.style={
		circle,
		fill=none,
		scale=0.5,
		font=\sffamily\bfseries\large,
	}
}

\tikzset{
	white-vertex-border/.style={
		circle,
		fill=black,
		scale=0.05,
		font=\sffamily\bfseries\large,
	}
}

\tikzset{
	vertex/.style={
		circle,
		fill=black!7.5!white,
		draw=black,
		line width=0.25pt,
		scale=0.6,
		font=\sffamily\bfseries\small,
		minimum size=1.5em, % Adjust this to your preference
		inner sep=0pt, % This ensures padding inside node is zero
	}
}

\tikzset{
	black-vertex/.style={
		circle,
		fill=black,
		scale=0.5,
		font=\sffamily\bfseries\large,
		text=white,
	}
}

\tikzset{
	blue-vertex/.style={
		circle,
		fill=blue!30,
		draw=blue!70!black,
		line width=1.5pt,
		scale=0.5,
	}
}

\tikzset{
	red-vertex/.style={
		rectangle,
		fill=red!30,
		draw=red!70!black,
		line width=1.5pt,
		scale=0.5,
		minimum width=0.1cm,
		minimum height=0.1cm,
	}
}

\tikzset{
	green-vertex/.style={
		circle,
		fill=green!30,
		draw=green!70!black,
		line width=1.5pt,
		scale=0.5,
	}
}

\tikzstyle{black-edge}=[-, draw=black]

\tikzstyle{blue-arc}=[->, draw=blue]

\tikzstyle{teal-arc}=[->, draw=teal]
\tikzstyle{cyan-arc}=[->, draw=cyan]

\tikzstyle{darkgray-arc}=[->, draw=darkgray, dotted]
\tikzstyle{gray-arc}=[->, draw=gray, dashed]
\tikzstyle{lightgray-arc}=[->, draw=lightgray]

\tikzstyle{dotted-edge}=[-, dotted]
\tikzstyle{densely-dotted-edge}=[-, densely dotted]
\tikzstyle{loosely-dotted-edge}=[-, loosely dotted]

\tikzstyle{dashed-edge}=[-, dashed]
\tikzstyle{densely-dashed-edge}=[-, densely dashed]
\tikzstyle{loosely-dashed-edge}=[-, loosely dashed]

\tikzstyle{arc}=[->, draw=purple!70, line width=1.5pt]
\tikzstyle{green-densely-dotted-arc}=[->, draw=green!50!black, line width=0.75pt]
\tikzstyle{red-loosely-dashed-arc}=[->, densely dashed, draw=red!50!black, line width=0.5pt]

\tikzset{
	densely-dotted-edge/.style={
		black, % Set the color of the edge (e.g., red!70)
		line width=0.5pt, % Set the thickness of the edge
		-, % Set the arrow style for the edge (-> for directed, - for undirected)
		densely dotted,
	}
}

\tikzset{
	blue-arc/.style={
		blue!50!black, % Set the color of the edge (e.g., red!70)
		line width=0.25pt, % Set the thickness of the edge
		->, % Set the arrow style for the edge (-> for directed, - for undirected)
	}
}

\tikzset{
	red-arc/.style={
		red!50!black, % Set the color of the edge (e.g., red!70)
		line width=0.5pt, % Set the thickness of the edge
		->, % Set the arrow style for the edge (-> for directed, - for undirected)
		densely dotted,
	}
}

\tikzset{
	yellow-arc/.style={
		yellow!50!black, % Set the color of the edge (e.g., red!70)
		line width=1pt, % Set the thickness of the edge
		->, % Set the arrow style for the edge (-> for directed, - for undirected)
	}
}

\tikzstyle{black-arc}=[->, draw=black, line width=0.5pt]

\tikzset{
	green-arc/.style={
		green!50!black, % Set the color of the edge (e.g., red!70)
		line width=0.5pt, % Set the thickness of the edge
		->, % Set the arrow style for the edge (-> for directed, - for undirected)
		densely dashed,
	}
}

\tikzstyle{purple-arc}=[->, thick, draw=purple, thick,dash dot]
\tikzstyle{orange-arc}=[->, draw=orange, ultra thick]
\tikzstyle{brown-arc}=[->, draw=brown,  ultra thick]

\definecolor{lightslategray}{rgb}{0.47, 0.53, 0.6}

\tikzstyle{none}=[inner sep=0pt]
\definecolor{hexcolor0xf81e1c}{rgb}{0.973,0.118,0.110}
\definecolor{hexcolor0x3c00ff}{rgb}{0.235,0.000,1.000}

\tikzstyle{whitevertex}=[circle,fill=white,draw=black, scale = 0.7]
\tikzstyle{redvertex}=[circle,fill=hexcolor0xf81e1c,draw=black, scale = 0.7]
\tikzstyle{bluevertex}=[circle,fill=hexcolor0x3c00ff,draw=black, scale = 0.7]
\tikzstyle{greenvertex}=[circle,fill=green,draw=black, scale=0.7]
\tikzstyle{purplevertex}=[circle,fill=magenta,draw=black, scale=0.7]
\tikzstyle{grayvertex}=[circle,fill=white,draw=gray, scale=0.7]
\tikzstyle{blackvertex}=[circle,fill=black,draw=black, scale=0.7]

\tikzstyle{textbox}=[rectangle,fill=none,draw=none]
\tikzstyle{box}=[rectangle,fill=none,draw=black]

\tikzstyle{grayedge}=[draw=gray]
\tikzstyle{blueedge}=[draw=blue]
\tikzstyle{rededge}=[draw=red]
\tikzstyle{edge}=[draw=black]

\newtheorem{theorem}{Theorem}

\newtheorem{claim}[theorem]{Claim}

\newtheorem{lemma}[theorem]{Lemma}

\newtheorem{proposition}[theorem]{Proposition}
\newtheorem{ques}{Question}
\newtheorem{problem}[ques]{Problem}

\newtheorem{obs}[theorem]{Observation}

\begin{document}

\title{\textbf{Bounding the Eviction Number}\\\textbf{of a Graph in Terms of its}\\\textbf{Independence Number\thanks{This paper is based on work that appears in
the doctoral dissertation\textquotedblleft Mobile Guards' Strategies for Graph
Surveillance and Protection\textquotedblright\ by Virg\'{e}lot Virgile, University of Victoria, 2024.}}}
\author{G. MacGillivray\thanks{Funded by a Discovery Grant from the Natural Sciences
and Engineering Research Council of Canada, RGPIN-04459-2017,
RGPIN-03930-2020.}, C. M. Mynhardt$^{\dag}$, V. Virgile\\Department of Mathematics and Statistics,\\University of Victoria, Victoria, \textsc{Canada}\\{\small gmacgill@uvic.ca, kieka@uvic.ca, virgilev@uvic.ca\medskip}\\\medskip To our friend and colleague, Odile Favaron}
\maketitle

\begin{abstract}
An eternal dominating family of graph $G$ in the eviction game is a collection
$\mathcal{D}_{k}=\{D_{1},...,D_{l}\}$ of dominating sets of $G$ such that (a)
$|D_{i}|=|D_{j}|$ for all $i,j\in\{1,2,...,l\}$, and (b) for any $i\in
\{1,2,...,l\}$ and any $v\in D_{i}$, either all neighbours of $v$ belong to
$D_{i}$, or there are a neighbour $w$ of $v$ not in $D_{i}$ and an integer
$j\in\{1,2,...,l\}\setminus\{i\}$ such that $D_{i}\cup\{w\}\setminus
\{v\}=D_{j}$. The eviction number of $G$, denoted by $e^{\infty}(G)$, is the
smallest cardinality of the sets in such an eternal dominating family. 

We compare $e^{\infty}$ to the independence number $\alpha$. We show
that the ratio $\alpha/e^{\infty}$ is unbounded and construct an
infinite class of connected graphs for which $e^{\infty}/\alpha \approx 4/3$. As our
main result, we use Ramsey numbers to show that for any integer $k\geq1$,
there exists a function $f(k)$ such that any graph with independence number
$k$ has eviction number at most $f(k)$.

\end{abstract}

\section{Introduction}

\label{SecIntro}Graph protection involves the placement of mobile sensors on
the vertices of a graph $G$ to protect the vertices and edges of $G$ against either
single or longer sequences of \textquotedblleft events\textquotedblright%
\ occurring at the vertices or edges. We refer to these sensors as \textit{guards}, and
to the events as \textit{attacks}. A guard located on a vertex $v$ of a graph $G$
covers $v$ and all the neighbours of $v$; we say that $v$ is \textit{occupied} (by a
guard). A vertex without a guard is said to be \textit{unoccupied}. The initial challenge is to
find a (usually smallest) subset $S$ of occupied vertices of $G$ such that
every vertex of $G$ is either in $S$ or adjacent to a vertex in $S$, i.e., a
\textit{dominating set} of $G$. However, real-world systems seldom remain static. As
situations evolve, the once optimal dominating set of guards may face
unforeseen challenges. Imagine a situation where a sensor must be moved from
its original position due to maintenance needs, adverse weather conditions, or
other complications. The challenge then becomes: 

\begin{center}
\textquotedblleft How do we reposition this sensor without
\vspace{-4pt}
compromising the overall coverage?\textquotedblright\ 
\end{center}

This is the point where the movement of guards becomes relevant. We restrict
our investigation to attack sequences of arbitrary length. Each such problem, called an \textbf{eternal domination problem}, can be modelled as a two-player game, alternating between a defender and an
attacker: the defender chooses the initial configuration of guards as well as
each configuration following an attack, and the attacker chooses the locations
of the attack. We further restrict our attention to the \textbf{eviction
game}, which was introduced by Klostermeyer, Lawrence and MacGillivray in
\cite{klostermeyer2016dynamic}. 

In the eviction game, only vertices containing a guard may be attacked. In the standard version of \emph{the eviction
game}, which we will simply refer to as \emph{eviction}, the guards
start by choosing their opening configuration, which must induce a dominating set. We consider 
the case where at most one guard is located on each vertex. At each turn the
attacker selects a vertex $v$ on which there is a guard and the guard on $v$ responds by moving 
to an unoccupied neighbour; that is, a neighbour of $v$ on which there is no guard, if possible.  If each of 
the neighbours of $v$ is occupied, then we say that $v$ is \emph{surrounded}, and the guard 
on $v$ must stay put and cannot contribute to the dominating set for a time unit. 
We sometimes say that an attacked guard is \emph{evicted}.
Only the guard that is attacked is allowed to move to a neighbour. The guards win the game 
if they are able to maintain a dominating set in the graph after responding to each attack;
otherwise, the attacker wins. Assuming the guards move optimally, an 
\emph{eternal dominating set} of a graph $G$ (in the eviction game) is any opening configuration 
(a dominating set of $G$) from which they can defend any sequence of attacks on $G$. 
The \emph{eviction number} of a graph $G$, denoted by $e^\infty(G)$, is the minimum 
cardinality of an eternal dominating set of $G$ in the eviction game.

We focus on comparing the eviction number of $G$ to its independence number
$\alpha(G)$. As we show in Section \ref{SecDefs}, it is easy to see that the ratio
$\alpha/e^{\infty}$ is unbounded. On the other hand, it is not so easy
to determine whether the ratio $e^{\infty}/\alpha$ is bounded or not.
The cycle $C_{7}$ is an example of a graph whose eviction number exceeds its
independence number: $\alpha(C_{7})=3$ and $e^{\infty}(C_{7})=4$. We construct an 
infinite class of connected graphs for which $e^{\infty
}/\alpha \approx 4/3$. One of the difficulties one encounters when studying eviction is the anomaly that $e^{\infty}$ could \textbf{increase} upon the addition of an edge. We illustrate this in Section \ref{different}.
As our main result, we show in Section \ref{functionf} that, for any integer
$k\geq1$, there exists a function, which we denote by $f(k)$, such that any
graph with independence number $k$ has eviction number at most $f(k)$. We state some open problems in Section \ref{open}.

\section{Definitions and Background}

\label{SecDefs}Concepts not defined here can be found in any standard text on
graph theory, e.g. \cite{CL, West}. For further background on graph protection, see \cite{klostermeyer2016protecting, klostermeyer2020eternal}.

For any positive integers $n$ and $k$, let
$[n]$ denote the set $\{1,2,...,n\}$ and let $\binom{[n]}{k}$ denote the set
of $k$-subsets of $[n]$.

We consider finite simple graphs and, as usual, denote the vertex and edge
sets of a graph $G$ by $V(G)$ and $E(G)$, respectively. The \emph{open
neighbourhood} of $v\in V(G)$, denoted by $N(v)$, is the set of vertices that
are adjacent to $v$. These vertices are known as the \emph{neighbours} of $v$.
The \emph{closed neighbourhood} of $v$ is the set $N[v]=N(v)\cup\{v\}$.

The \emph{disjoint union} of a graph $G$ and a graph $H$, denoted by $G+H$, is the graph with vertex set $V(G) \cup V(H)$ and edge set $E(G) \cup E(H)$. We denote the disjoint union of $k$ disjoint copies of a graph $G$ by $k G$. The \emph{join} of a graph $G$ and a graph $H$, denoted by $G \vee H$, is the graph with vertex set $V(G) \cup V(H)$ and edge set $E(G) \cup E(H) \cup \{uv: u \in V(G), v \in V(H)\}$.

As stated above, we denote the independence number of $G$ by $\alpha(G)$.
Furthermore, we denote the clique number of $G$ by $\omega(G)$, the domination
number by $\gamma(G)$, the chromatic number by $\chi(G)$, and the clique
covering number (the minimum cardinality of a partition of $V(G)$ such that
each set in the partition induces a clique) by $\theta(G)$. Observe that
$\alpha(G)=\omega(\overline{G})$ and $\chi(G)=\theta(\overline{G})$ for 
any graph $G$.

Let $\mathcal{D}_{k}$ be the collection of all dominating sets of $G$ of fixed
cardinality $k$. For $D\in\mathcal{D}_{k}$, we imagine that there is a single guard located on
each vertex of $D$ and therefore we think of $D$ as a configuration of guards.
We say that a (not necessarily dominating) set $X$ \emph{protects} a vertex
$v$, or $v$ is \emph{protected }(by $X$), if $v$ or one of its neighbours is occupied by a member of $X$.

As explained in the introduction, each eternal domination problem can be
modelled as a two-player game, alternating between a \emph{defender} and an
\emph{attacker}: the defender chooses $D_{1}\in\mathcal{D}_{k}$ as well as
each $D_{i}$, $i>1$, while the attacker chooses the locations $r_{1}%
,r_{2},\ldots$ of the attacks; we say the attacker \emph{attacks} the vertices
$r_{i}$. Thus, the game starts with the defender choosing $D_{1}$. For
$i\geq1$, the attacker attacks $r_{i}$ and the defender \emph{defends against}
the attack by choosing $D_{i+1}\in\mathcal{D}_{k}$ subject to constraints that
depend on the particular game. The defender wins the game if they can
successfully defend the graph against any sequence of attacks, including
sequences that are infinitely long, subject to the constraints of the game;
the attacker wins otherwise. In other words, the attacker's goal is to force
the defender into a configuration of guards that is not dominating. These
 dynamic models of domination were first
defined and studied by Burger, Cockayne, Gr\"{u}ndlingh, Mynhardt, Van Vuuren
and Winterbach in \cite{BCG-S, BCG2}. In particular, they studied the eternal 
domination number $\gamma^\infty(G)$ of a graph $G$, which is the smallest number of guards
that can defend $G$ against arbitrary sequences of attacks on unguarded vertices, where a single guard must move to the attacked vertex.

The definitions below of eternal dominating families of a graph in the eternal domination and the eviction games illustrate the difference between the two protection models.

An \emph{eternal dominating family} of a graph $G$ (in the eternal domination game) is a collection of sets $D_1, D_2, D_3, \ldots, D_l$ of $G$ that satisfy the following properties.

\begin{enumerate}
	\item For any $i, j \in [l]$, $|D_i|=|D_j|$.
    
	\item For any $i \in [l]$ and any $w \in V(G)-D_i$, there exist $v \in D_i \cap N(w)$ and $j \in [l]-\{i\}$ such that $(D_i \cup \{w\})-\{v\} = D_j$.
\end{enumerate}

An \emph{eternal dominating family} of a graph $G$ (in the eviction game) is a
collection of dominating sets $D_{1},D_{2},D_{3},\ldots,D_{l}$ of $G$ that satisfy the
following properties.

\begin{enumerate}
\setcounter{enumi}{2}
\item For any $i, j \in[l]$, $|D_{i}|=|D_{j}|$. 

\item For any $i\in\lbrack l]$ and any $v\in D_{i}$, either $N[v]\subseteq
D_{i}$, or there exist $w\in N(v)-D_{i}$ and $j\in\lbrack l]-\{i\}$ such that
$D_{i}\cup\{w\}-\{v\}=D_{j}$.
\end{enumerate}

Item (2) implies that the $D_i$ are dominating sets whereas (4) does not; thus we have to specify this explicitly. We proceed by stating some results on eviction obtained by Klostermeyer, Lawrence and MacGillivray in
\cite{klostermeyer2016dynamic}.

\begin{proposition}[\cite{klostermeyer2016dynamic}]
If $k < |V(G)|$ guards can defend an arbitrarily long sequence of attacks in the eviction game on a graph $G$, then so can $k+1$ guards.  
\end{proposition}  \label{Proposition:EvictionMoreGuards}

We now examine how the eviction number of a graph $G$ relates to some other parameters such as the domination number, the independence number and the clique covering number of $G$.

\begin{proposition} [\cite{klostermeyer2016dynamic}] \label{Proposition:EvictionBasicBounds}
	For any graph $G$, $\gamma(G) \leq e^\infty(G) \leq \theta(G)$.
\end{proposition}

Since $\alpha(G)$ is a lower bound on $\gamma^\infty(G)$ (see \cite{BCG2}) and belongs to the interval $[\gamma(G), \theta(G)]$, it is reasonable to ask whether $\alpha(G)$ is also a lower bound on $e^\infty(G)$. Observation \ref{Observation:EvictionNumberOne} shows that this is not the case. Indeed, while fixing $e^\infty(G)=1$, $\alpha(G)$ can be arbitrarily large. Hence the ratio $\alpha/e^\infty$ is unbounded.

\begin{proposition}
	Let $G$ be a graph with at least two universal vertices. Then $\left. e^{\infty }(G)=1\right..$ % ${e^\infty(G)=1}$.
\end{proposition}

\begin{proof}
	Let $u, v$ be two universal vertices of $G$. By moving back and forth on the vertices $u$ and $v$, one guard can dominate all of the vertices of $G$ at each time $t=1,2,3,\ldots.$
\end{proof}

\begin{obs} \label{Observation:EvictionNumberOne}
	Let $G$ be the join of the graph $K_2$ with the graph $\overline{K_m}$ $($see Figure $\ref{Figure:EvictionNumberOne})$. Then $\alpha(G)=m$ and $e^\infty(G)=1$.
    
\end{obs}

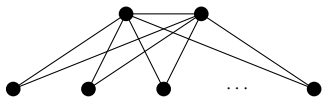
\begin{figure}[htbp]
	\centering
	\begin{tikzpicture}
		\begin{pgfonlayer}{nodelayer}
			\node [style=black-vertex] (0) at (-0.5, 0) {};
			\node [style=black-vertex] (1) at (0.5, 0) {};
			\node [style=black-vertex] (2) at (-2, -1) {};
			\node [style=black-vertex] (3) at (-1, -1) {};
			\node [style=black-vertex] (4) at (0, -1) {};
			\node [style=white-vertex] (5) at (1, -1) {$\cdots$};
			\node [style=black-vertex] (6) at (2, -1) {};
		\end{pgfonlayer}
		\begin{pgfonlayer}{edgelayer}
			\draw (0.center) to (2.center);
			\draw (2.center) to (1.center);
			\draw (1.center) to (6.center);
			\draw (6.center) to (0.center);
			\draw (0.center) to (3.center);
			\draw (3.center) to (1.center);
			\draw (1.center) to (4.center);
			\draw (4.center) to (0.center);
			\draw (0.center) to (1.center);
		\end{pgfonlayer}
	\end{tikzpicture}

	\caption{$K_2 \vee \overline{K_m}$}

	\label{Figure:EvictionNumberOne}
\end{figure}

It is not so easy to determine whether the ratio $e^{\infty}/\alpha$ is bounded or not. The cycle $C_{7}$ is an example of a graph with $\alpha=3$ and $e^{\infty}=4$. Therefore, disjoint unions of $C_{7}$ provide infinitely many (disconnected) graphs $G$ for which $e^{\infty}(G)/\alpha(G)=\frac{4}{3}$. To see that a similar result holds for connected graphs, consider the graph $G_k$ obtained by joining a new vertex $v$ to each vertex of $kC_7$, and a new vertex $w$ to $v$ (see Figure \ref{Figure:RatioLimit} for the case where $k=2$). The graph $G_k$ satisfies the properties described in the following proposition.

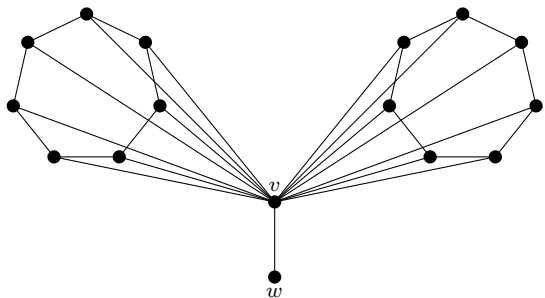
\begin{figure}[!htbp]
	\centering
	
	\begin{tikzpicture}
		\begin{scope}[rotate=-1*90/7]
			\foreach \x in {0, 1, ..., 6}
			{
				\draw ({\x*360/7}:1) -- ({(\x+1)*360/7}:1);
				\fill ({\x*360/7}:1) circle[radius=2.5pt];
			}
		\end{scope}
		
		\begin{scope}[shift={(5,0)}]
			\begin{scope}[rotate=-1*90/7]
				\foreach \x in {0, 1, ..., 6}
				{
					\draw ({\x*360/7}:1) -- ({(\x+1)*360/7}:1);
					\fill ({\x*360/7}:1) circle[radius=2.5pt];
				}
			\end{scope}
		\end{scope}
		
		\begin{scope}[shift={(2.5,-1.5)}]
			\fill[black] (0*360/7:0) circle[radius=2.5pt] node[above] {\scriptsize $v$};
		\end{scope}
		
		\begin{scope}[shift={(2.5,-2.5)}]
			\fill[black] (0*360/7:0) circle[radius=2.5pt] node[below] {\scriptsize $w$};
		\end{scope}
		
		\draw (2.5,-1.5)+(0*360/7:0) -- (2.5,-2.5)+(0*360/7:0);
		
		\foreach \x in {0, 1, ..., 6}
		{
			\draw ({-1*90/7+\x*360/7}:1) -- (2.5,-1.5)+(0*360/7:0);
		}
		
		\foreach \x in {0, 1, ..., 6}
		{
			\draw (5,0)+({-1*90/7+\x*360/7}:1) -- (2.5,-1.5)+(0*360/7:0);
		}
	\end{tikzpicture}	
	
	\caption{Graph $G_k$ for the case where $k=2$.}
	\label{Figure:RatioLimit}
\end{figure}

\begin{proposition} \label{4/3}
	For any $k \geq 1$, $\alpha(G_k)=3k+1$ and $e^\infty(G_k)=\theta(G_k)=4k+1$.
\end{proposition}

\begin{proof}
	The reader can easily verify that $\alpha(G_k)=3k+1$ and $\theta(G_k)=4k+1$. So, by Proposition \ref{Proposition:EvictionBasicBounds}, we only need to show that $e^\infty(G_k) \geq 4k+1$. Suppose the assumption is false. Then $G_k$ can be defended by $4k$ guards (by Proposition \ref{Proposition:EvictionMoreGuards}). Consider a configuration of these $4k$ guards on $G_k$. If $w$ is occupied and $v$ is unoccupied, then evict the guard on $w$ to $v$. Otherwise, $v$ and $w$ are both occupied; in this case, evict the guard on $v$ to an unoccupied vertex of one of the copies of $C_7$ and then evict the guard on $w$ to $v$. So, we may assume without loss of generality that $v$ is occupied and $w$ is unoccupied. Since there are $4k$ guards located on $G_k$, there is a copy of $C_7$, which we will refer to as $C_7^*$, on which are located fewer than four guards. Evict the guards located on $C_7^*$ in a way such that a vertex of $C_7^*$ is not dominated by any guard located on a vertex of this subgraph. Now, evict the guard located on $v$. If the guard moves to $w$, then a vertex of $C_7^*$ is not dominated. If the guard moves somewhere else, then $w$ is not dominated. This contradicts our assumption that $4k$ guards can defend the graph.
\end{proof}

However, $\alpha(G)$ is a lower bound on $e^\infty(G)$ when $G$ belongs to a specific graph class, as stated below.

\begin{proposition}[\cite{klostermeyer2016dynamic}]
 If $G$ is a triangle-free graph, then $e^\infty(G) \geq \alpha(G)$.   
\end{proposition} 

We now give the values of $e^\infty$ for paths, cycles and complete bipartite graphs.

\begin{proposition} [\cite{klostermeyer2016dynamic}] \label{Proposition:EvictionBasicFamilies}
	For any integers $n, m \geq 1$,
	\begin{enumerate} [label=(\roman*), left=0pt, rightmargin=0pt, topsep=0pt, itemsep=0pt]
		\item $e^\infty(P_n)=\lceil \frac{n}{2} \rceil$,
		\item $e^\infty(C_3)=1$, $e^\infty(C_5)=2$ and $e^\infty(C_n)=\lceil \frac{n}{2} \rceil$ for any $n \neq 3, 5$,
		\item $e^\infty(K_{m, n})=\max \{m, n\}$.
	\end{enumerate}
\end{proposition}

Although $\alpha(G)$ is neither an upper bound nor a lower bound on $e^\infty(G)$, Klostermeyer and MacGillivray show that $e^\infty(G)$ is bounded for the first three values of $\alpha$.

\begin{theorem} [\cite{klostermeyer2016dynamic}]  \label{small} If
\[
\alpha(G)=\left\{
\begin{tabular}
[c]{l}%
$1$\\
$2$\\
$3$%
\end{tabular}
\right.  \text{,\ then\ }e^{\infty}(G)\left\{
\begin{tabular}
[c]{l}%
$=1$\\
$\leq2$\\
$\leq5.$%
\end{tabular}
\right.
\]

\end{theorem}

It is unknown whether there exists a graph $G$ such that $\alpha(G)=3$ and $e^{\infty}(G)=5$. There are also no results in the literature bounding the eviction number in terms of the independence number when the latter is at least $4$.

Klostermeyer and MacGillivray \cite{klostermeyer2007eternal} proved that
$\gamma^{\infty}(G)\leq\binom{\alpha(G)+1}{2}$ for any graph $G$. Thus, the eternal domination number of a graph is bounded by a function of its independence number. Since no such bound is known in general for the eviction number of the graph, Klostermeyer, Lawrence and MacGillivray asked the following questions.

\begin{ques} [\cite{klostermeyer2016dynamic}] \label{Question:Eviction_c*alpha}
	Does there exist a constant $c$ such that $e^\infty(G) \leq c \alpha(G)$ for all graphs $G$?
\end{ques}

\begin{ques} [\cite{klostermeyer2016dynamic}]
	Does there exist a graph $G$ such that $\gamma^\infty(G)<e^\infty(G)$?
\end{ques}

The results in the next two sections are motivated by Question \ref{Question:Eviction_c*alpha}, which still remains unanswered. We aim to show the existence of a function $f$ such that any graph with independence number $k$ has eviction number at most $f(k)$. We begin by illustrating one of the difficulties we encountered when trying to determine the eviction number of a graph by considering its subgraphs.

\section{Eviction is Different}
\label{different}

For almost any domination-type parameter $\pi$, adding an edge to a graph $G$ can only result in a graph $G^{\prime}$ with $\pi(G^{\prime}) \leq \pi(G)$, and usually $\pi(G^{\prime}) \in \{\pi(G), \pi(G)-1\}$. This is not the case with the eviction number. 

Consider the graph $G \cong K_{1} + (K_{2} \vee \overline{K_{t}}),t\geq2$. Observe that $G$ has eviction number $2$ and the graph $G'$ obtained from $G$ by adding an edge from the isolated vertex of $G$ to one of the vertices of degree $t+1$ has eviction number $t+1$. The problem occurs when the attacker can force a guard to be surrounded.

This is trivially the case for $K_{1}$, but there are infinitely many graphs with this property. For example, the spider $\operatorname{Sp}(2;k)$, which is obtained from the star $K_{1,k}$ by subdividing each edge exactly once, has eviction number $k+1$. The attacker can force guards to be on the central vertex $c$ and all its neighbours.
Now, when $c$ is joined to a vertex of another graph, for example $K_{2}\vee\overline{K_{t}}$ (and again there are infinitely many examples), the eviction number of the resulting graph can be arbitrarily higher than the eviction number of $\operatorname{Sp}(2;k) + (K_{2}\vee\overline{K_{t}})$ (see Figure \ref{Figure:EvictionEdgeAdditionRemovalDisconnected}).

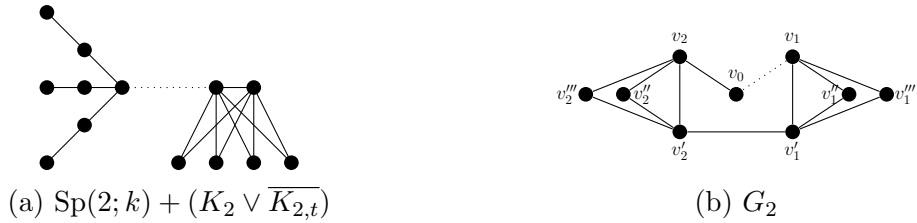
\begin{figure}[htbp]
	\centering
	\begin{subfigure}[b]{0.45\textwidth}
		\centering
		\begin{tikzpicture}
			\begin{pgfonlayer}{nodelayer}
				\node [style=black-vertex] (0) at (-1, 0) {};
				\node [style=black-vertex] (1) at (-1.5, 0.5) {};
				\node [style=black-vertex] (2) at (-2, 1) {};
				\node [style=black-vertex] (3) at (-1.5, 0) {};
				\node [style=black-vertex] (4) at (-2, 0) {};
				\node [style=black-vertex] (5) at (-1.5, -0.5) {};
				\node [style=black-vertex] (6) at (-2, -1) {};
				\node [style=black-vertex] (7) at (0.25, 0) {};
				\node [style=black-vertex] (8) at (0.75, 0) {};
				\node [style=black-vertex] (9) at (-0.25, -1) {};
				\node [style=black-vertex] (10) at (0.25, -1) {};
				\node [style=black-vertex] (11) at (0.75, -1) {};
				\node [style=black-vertex] (12) at (1.25, -1) {};
			\end{pgfonlayer}
			\begin{pgfonlayer}{edgelayer}
				\draw (7.center) to (9.center);
				\draw (9.center) to (8.center);
				\draw (8.center) to (10.center);
				\draw (10.center) to (7.center);
				\draw (7.center) to (11.center);
				\draw (11.center) to (8.center);
				\draw (8.center) to (12.center);
				\draw (12.center) to (7.center);
				\draw (7.center) to (8.center);
				\draw (0.center) to (3.center);
				\draw (3.center) to (4.center);
				\draw (0.center) to (1.center);
				\draw (1.center) to (2.center);
				\draw (0.center) to (5.center);
				\draw (5.center) to (6.center);
				\draw [style=dotted-edge] (0.center) to (7.center);
			\end{pgfonlayer}						
		\end{tikzpicture}

		\caption{$\operatorname{Sp}(2;k) + (K_2 \vee \overline{K_{2,t}})$}
		\label{Figure:EvictionEdgeAdditionRemovalDisconnected}	
	\end{subfigure}
	\hfill
	\begin{subfigure}[b]{0.45\textwidth}
		\centering
		\begin{tikzpicture}
			\begin{pgfonlayer}{nodelayer}
				\node [style=black-vertex] (0) at (-0.75, 0) {};
				\node [style=black-vertex] (1) at (0, -0.5) {};
				\node [style=black-vertex] (2) at (0.75, 0) {};
				\node [style=black-vertex] (3) at (-2, -0.5) {};
				\node [style=black-vertex] (4) at (-1.5, -0.5) {};
				\node [style=black-vertex] (5) at (-0.75, -1) {};
				\node [style=black-vertex] (6) at (0.75, -1) {};
				\node [style=black-vertex] (9) at (2, -0.5) {};
				\node [style=black-vertex] (10) at (1.5, -0.5) {};
				\node [style=white-vertex] (11) at (0.75, 0.25) {$v_1$};
				\node [style=white-vertex] (12) at (-0.75, 0.25) {$v_2$};
				\node [style=white-vertex] (13) at (0, -0.25) {$v_0$};
				\node [style=white-vertex] (14) at (0.75, -1.25) {$v_1'$};
				\node [style=white-vertex] (15) at (-0.75, -1.25) {$v_2'$};
				\node [style=white-vertex] (16) at (2.25, -0.5) {$v_1'''$};
				\node [style=white-vertex] (17) at (-2.25, -0.5) {$v_2'''$};
				\node [style=white-vertex] (18) at (-1.25, -0.5) {$v_2''$};
				\node [style=white-vertex] (19) at (1.25, -0.5) {$v_1''$};
			\end{pgfonlayer}
			\begin{pgfonlayer}{edgelayer}
				\draw (3) to (0);
				\draw (0) to (5);
				\draw (5) to (4);
				\draw (4) to (0);
				\draw (3) to (5);
				\draw (0) to (1);
				\draw (2) to (6);
				\draw (2) to (10);
				\draw (10) to (6);
				\draw (6) to (9);
				\draw (9) to (2);
				\draw [style=dotted-edge] (2) to (1);
				\draw (5) to (6);
			\end{pgfonlayer}
		\end{tikzpicture}
		
		\caption{$G_2$}
		\label{Figure:EvictionEdgeAdditionRemovalConnected}	
	\end{subfigure}

	\caption{Examples of graphs whose eviction number increases by the addition of one edge.}
	\label{Figure:EvictionEdgeAdditionRemoval}
\end{figure}

The graph $G_2$ in Figure \ref{Figure:EvictionEdgeAdditionRemovalConnected} is an example of a connected graph for which adding an edge increases the eviction number. In fact, three guards can defend the subgraph induced by the vertices $v_0, v_2, v_2', v_2'', v_2'''$ using the same strategy from the previous paragraph while one guard defends the subgraph induced by the vertices $v_1, v_1', v_1'', v_1'''$. However, the graph $G_2'$ obtained from $G_2$ by the addition of the edge $v_0v_1$ cannot be defended by four guards. To see this, suppose four guards are initially located on the vertices of $G_2'$. We may assume without loss of generality that the vertex $v_0$ is occupied since its neighbourhood is an independent set. We may further assume without loss of generality that there are two guards located on $v_2$ and $v_2'$, and one guard located on $v_1'$. After the sequence of attacks $v_0, v_1', v_1, v_2', v_2$, a vertex of $G_2'$ is not dominated.

A graph may have several vertices on which guards are surrounded at the same time. By Proposition  \ref{Proposition:EvictionBasicFamilies}, the complete bipartite graph $K_{r, r+s}$, where $s>0$, has eviction number $r+s$. The attacker can force $r$ guards to be located on the vertices of degree $r+s$ and $s$ guards to be located on the vertices of degree $r$; these $s$ guards are all surrounded.

The above paragraphs illustrate that, when determining the eviction number of a graph $G$ by considering how guards can defend various subgraphs of $G$, \textbf{it may be important to establish that the guards can always move within these specific subgraphs.}

\section{The Function $f$}
\label{functionf}
We begin our result on the function $f$ with the following definition.

The \emph{Ramsey number} $r(k, l)$ is the minimum integer $n$ such that any graph on at least $n$ vertices contains either an independent set of size $k$ or a clique of size $l$.

The special case of Ramsey's Theorem \cite{ramsey1929logic} applied to a graph and its complement ensures that the Ramsey number $r(k, l)$ is well defined. We use this theorem to prove the following lemma on which the proof of our main theorem depends.

\begin{lemma} \label{Lemma:MaxIndSets}
	For any integer $k \geq 1$, there exists a constant $c(k)=c_k$ such that if $G$ is a graph with independence number $k$ which has at least $c_k$ disjoint maximum independent sets, then $G$ can be defended by $k^2$ guards. Moreover, no guard is ever prevented from moving by being surrounded. 

\end{lemma}

\begin{proof}
	Let $k \geq 1$ and let $l_0, l_1, l_2, l_3, \ldots, l_k$ be sufficiently large positive integers (which will be chosen later). Let $G$ be a graph with independence number $k$. Suppose $G$ has at least $c_k = l_0$ disjoint maximum independent sets $R_1, R_2, R_3, \ldots, R_{c_k}$, where $R_i=\{v_{i,1}, v_{i,2}, v_{i,3}, \ldots, v_{i,k}\}$ for each $i=1,2,3,\ldots,c_k$.

	If $l_0 \geq r(k+1, l_1)$, Ramsey's Theorem guarantees that there exist $l_1$ vertices in the set $\bigcup_{i \in [l_0]} \{v_{i,1}\}$ which induce a complete subgraph of $G$. Without loss of generality, let $C_1=\{v_{1,1}, v_{2,1}, v_{3,1}, \ldots, v_{l_1,1}\}$ be such a set.	
	If $l_1 \geq r(k+1, l_2)$, Ramsey's Theorem guarantees that there exist $l_2$ vertices in the set $\bigcup_{i \in [l_1]} \{v_{i,2}\}$ which induce a complete subgraph of $G$. Without loss of generality, let $C_2=\{v_{1,2}, v_{2,2}, v_{3,2}, \ldots, v_{l_2,2}\}$ be such a set. 
	Likewise, for each $j \in \{3, 4, \ldots, k-1\}$, if $l_{j-1} \geq r(k+1, l_j)$, Ramsey's Theorem guarantees that there exist $l_j$ vertices in the set $\bigcup_{i \in [l_{j-1}]} \{v_{i,j}\}$ which induce a complete subgraph of $G$.
	Finally, if $l_{k-1} \geq r(k+1, k+1)$, then there exist $l_k = k+1$ vertices in the set $\bigcup_{i \in [l_{k-1}]} \{v_{i,k}\}$ which induce a complete subgraph of $G$.
	To summarize, let: 	
	
	\vspace{-24pt}
	\begin{alignat*}{3}
		& l_k &&= k+1 \\
		& l_{k-1} &&= r(k+1, l_k) &&= r(k+1, k+1) \\
		& l_{k-2} &&= r(k+1, l_{k-1}) &&= r(k+1, r(k+1, k+1)) \\
		&&\vdots \\
		& l_0 &&= r(k+1, l_1) &&= \underbrace{r(k+1, r(k+1, r(k+1, \ldots)))}_{\text{$k$ times}}.
	\end{alignat*}

	As explained above, if $G$ has at least $c_k = l_0$ disjoint independent sets of size $k$, then $G$ has a subset of vertices $S^*=\bigcup_{i \in [k+1], j \in [k]} \{v_{i,j}\}$, where:
	\begin{enumerate}[label=(\roman*), align=left]
		\item $R_i = \bigcup_{j \in [k]} \{v_{i,j}\}$ is a maximum independent set for each $i \in [k+1]$,
		\item $C_j' = \bigcup_{i \in [k+1]} \{v_{i,j}\}$ is a clique of size $k+1$ for each $j \in [k]$.
	\end{enumerate}
	
	We can represent $S^*$ as an array with rows $R_i, i \in [k+1]$, and columns $C_j', j \in [k]$, as below:
	
	\vspace{-12pt}
	\[
	\begin{bmatrix}
		v_{1,1} & v_{1,2} & v_{1,3} & \cdots & v_{1,k} \\
		v_{2,1} & v_{2,2} & v_{2,3} & \cdots & v_{2,k} \\
		\vdots & \vdots & \vdots & \ddots & \vdots \\
		v_{k+1,1} & v_{k+1,2} & v_{k+1,3} & \cdots & v_{k+1,k}
	\end{bmatrix}.
	\]

	In this case, place $k$ guards in $C_j'$ for each $j \in [k]$ so that the subset $S^*$ contains exactly $k^2$ vertices. For any $j \in [k]$, if a guard in $C_j'$ is attacked, the guard can always relocate to the only unoccupied vertex in $C_j'$. Since there are exactly $k+1$ rows $R_1, R_2, R_3, \ldots, R_{k+1}$ and there are exactly $k^2$ guards located in $S^*=\bigcup_{i \in [k+1]} R_i$ at each time $t=1,2,3,\ldots$, by the Generalized Pigeonhole Principle, there is an integer $m \in [k+1]$ such that row $m$ contains at least $\lceil \frac{k^2}{k+1} \rceil=k$ guards; that is, all the vertices in $R_m$ are occupied. Since $R_m$ is a maximum independent (and hence dominating) set of $G$, all the vertices in $G$ are dominated by $R_m$. This completes the proof.
\end{proof}

Observe that our proof of Lemma \ref{Lemma:MaxIndSets} shows that $c_1 \leq r(2,2)=2$ and $c_2 \leq r(3, r(3,3))=18$.

We are now ready to prove our main theorem. 

\begin{theorem} \label{Theorem:BoundEviction}
	There exists a function $f$ such that if $G$ is
a graph with independence number $k\geq1$, then $e^{\infty}(G)\leq f(k)$. In
particular,
\[
f(1)=1\text{\ and}\ f(k)\leq\frac{2kc_{k}(k^{k-1}-1)}{k-1}\ \text{when}%
\ k\geq2,
\]
where $c_{k}$ is as in Lemma \ref{Lemma:MaxIndSets}.

\end{theorem}

\begin{proof} 
	The cases $k=1, 2, 3$ are clear (see Theorem \ref{small}).
	Let $G$ be a graph such that $\alpha(G) = k \geq 4$. 
	We may assume that $|V| > f(k)$; otherwise, the theorem clearly holds for $G$.

	If $G$ has at least $c_k$ disjoint independent sets of size $k$, then, by Lemma \ref{Lemma:MaxIndSets}, $G$ can be defended by $k^2$ guards. So, we may further assume that $G$ has fewer than $c_k$ disjoint maximum independent sets. We first prove the following claim:

	{
		\setlength{\parskip}{0pt}
		\begin{quote}
			\begin{claim}
				There exist a positive integer $l < k$ and a subset $S$ of vertices of $G$ such that $\alpha(G-S)=l$ and $G-S$ has at least $c_l + |S|$ disjoint independent sets of size $l$.
			\end{claim}
			
			\bigskip
			
			\begin{proof}			
				Let $G_0, G_1, G_2, \ldots, G_{k-1}$ be a sequence of subgraphs of $G$ (where $G=G_0$) that satisfy the following conditions for each $i \in \{0, 1, 2, \ldots, k-2\}$:
				\begin{enumerate} [label=(\roman*), align=left]
					\item $\alpha(G_i) = k-i$.
					\item $G_{i+1} = G_i - S_i$, where $S_i$ is a smallest subset of vertices of $G_i$ such that $\alpha(G_i-S_i)= k-i-1$.

				\end{enumerate}
                \smallskip

				Since $G_0$ has fewer than $c_k$ disjoint independent sets of cardinality $k$, we have $|S_0| < k c_k$. 
                
				\smallskip
                
				If $G_{1}$ has at least $c_{k-1}+|S_{0}|$ disjoint independent sets of cardinality $k-1$, then we are done, hence suppose this is not the case. Then 
                \[
                |S_{1}| < (k-1)(c_{k-1}+|S_{0}|) < (k-1)(c_{k-1}+k c_k) < k^2 c_k.
                \]

				If $G_{2}$ has at least $c_{k-2}+|S_0|+|S_1|$ disjoint independent sets of cardinality $k-2$, then we are done, hence suppose this is not the case. Then 
                \[
                |S_{2}| < (k-2)(c_{k-2}+|S_0|+|S_1|) < (k-2)(c_{k-2}+k c_k+k^2 c_k) < k^3 c_k. 
                \]
                				                
				Likewise, for each $i \in \{3, 4, 5, \ldots, k-2\}$, if $G_{i}$ has at least $c_{k-i}+|S_0|+|S_1|+\cdots+|S_{i-1}|$ disjoint independent sets of cardinality $k-i$, then we are done, hence suppose this is not the case. Then 
                \begin{align*}
                  |S_{i}| & < (k-i)(c_{k-i}+|S_0|+|S_1|+\cdots+|S_{i-1}|) \\
                & < (k-i)(c_{k-i} + k c_k + k^2 c_k + \cdots + k^i c_k) \\
                & < k^{i+1} c_k.   
                \end{align*}

				Since $G$ is a graph of order at least $1 + \frac{2 k c_k (k^{k-1}-1)}{(k-1)}$, $G$ has a subset $S = \bigcup_{i=0}^{k-2} S_i$ of cardinality at most
                \begin{align*}
                 |S_0|+|S_1|+\cdots+|S_{k-2}| & < c_k (k + k^2 + k^3 + \ldots + k^{k-2} + k^{k-1}) \\
                 & = \frac{k c_k (k^{k-1}-1)}{k-1}   
                \end{align*}
                 such that $G-S$ is a graph on at least $1+|S|$ vertices with independence number $1$. This completes the proof of the claim.
				\renewcommand\qedsymbol{\ensuremath{\diamondsuit}}
			\end{proof}
		\end{quote}
	}

	\noindent Now, consider a smallest such subset of vertices $S$ such that $\alpha(G-S)=l < k$ and $G-S$ has at least $c_l + |S|$ disjoint independent sets of size $l$.

	Let $M$ be a smallest matching in $G$ that covers the largest number of vertices in $S$. Let $S'$ be the set of vertices in $G-S$ that belong to an edge of the matching $M$. Since $|S'| \leq |S|$, $\alpha(G-(S \cup S'))=l$ and $G-(S \cup S')$ has at least $c_l$ disjoint independent sets of size $l$. As a consequence of Lemma \ref{Lemma:MaxIndSets}, $l^2$ guards (and therefore at most $k^2$ guards) can defend $G-(S \cup S')$ and any guard that is evicted can always move to an unoccupied vertex in that subgraph.
	
	Since $|S \cup S'| < \frac{2 k c_k (k^{k-1}-1)}{k-1}$, in the rest of the proof, we will show that there exists a strategy with no more than $\frac{k c_k (k^{k-1}-1)}{k-1}$ guards to defend $G'$, the subgraph of $G$ induced by $S \cup S'$, where it is always possible for any guard located on $G'$ to move at each step of the game to a vertex of $G'$. Let the initial configuration of the guards be such that there is exactly one guard on each edge of $M$ and one guard on each vertex that is not covered by $M$, where such a vertex necessarily belongs to $S$. We will maintain the invariant that there is a guard on exactly one vertex of each edge of a matching $M$ that covers the largest number of vertices of $G'$ and one guard on each of the vertices that are not covered by the matching. The invariant is clearly initially satisfied. Suppose at some time $t=1,2,3,\ldots$ a guard located on $x_i$, a vertex that is covered by $M$, is attacked. Observe that the guard located on $x_i$ has at least one unoccupied neighbour $y_i$, which is matched to $x_i$ by $M$, since there is only one guard on each edge of the matching. Move the guard to its neighbour $y_i$. The invariant obviously still holds. Now, suppose a guard located on a vertex $z_i$ in $G'$ that is not covered by $M$ is attacked. Note that our choice of $M$ implies that $z_i \in S$. We consider two cases:
	
	\item Case $1$: If $z_i$ has no unoccupied neighbour in $G$, then there is nothing to do.
	
	\item Case $2$: If $z_i$ has an unoccupied neighbour $x_i$, then $x_i$ is covered by $M$ and therefore belongs to $G'$; otherwise we could find a matching that covers more vertices of $S$. Then, there exists $y_i \in S$ such that $x_i$ is matched to $y_i$ by $M$ and there is a guard on $y_i$. In this case, move the guard on $z_i$ to $x_i$ and consider the new matching $M'=(M \cup \{x_iz_i\}) \backslash \{x_i y_i\}$. The invariant clearly still holds.

Since $G$ is a graph on at least $1+\frac{2 k c_k (k^{k-1}-1)}{k-1}$ vertices, $V(G)$ can be partitioned into two sets $V_1$ and $V_2$ in a way such that at most $k^2 < \frac{k c_k (k^{k-1}-1)}{k-1}$ guards can effectively defend the subgraph of $G$ induced by $V_1$ and at most $\frac{k c_k (k^{k-1}-1)}{k-1}$ guards can effectively defend the subgraph induced by $V_2$. This completes the proof.
\end{proof}

\section{Open Problems}
\label{open}

As shown by Klostermeyer and MacGillivray in \cite{klostermeyer2016dynamic}, if $G$ is a graph with independence number $3$, then $G$ has eviction number at most $5$. However, it is unknown whether there exists a graph $G$ such that $\alpha(G)=3$ and $e^{\infty}(G)=5$. This naturally gives rise to the following question.

\begin{ques} \label{Question:Conclusion35}
	Does there exist a graph $G$ such that $\alpha(G)=3$ and $e^\infty(G)=5$?
\end{ques}

In Proposition \ref{4/3} we constructed an infinite class of connected graphs for which $e^{\infty}/\alpha \approx4/3$. We do not know of any graph for which $e^{\infty}/\alpha > 4/3$. The next question is more general than Question \ref{Question:Conclusion35}.

\begin{ques} 
	Does there exist a graph such that $e^{\infty}/\alpha > 4/3$?
\end{ques}

The upper bound on the function $f(k)$ in Theorem \ref{Theorem:BoundEviction} is much larger than $4/3$, the largest known value of $e^{\infty}/\alpha$, and almost certainly excessively large.

\begin{problem}
    (Substantially) improve the bound on the function $f(k)$ given in Theorem \ref{Theorem:BoundEviction}.
\end{problem}

A \textit{cograph} (or \textit{complement reducible graph}) is a graph that can be generated from the trivial graph $K_1$ by complementation and disjoint union. These graphs are also known under various characterizations, among which are the following.

\begin{proposition} [\cite{corneil1981complement,corneil1985linear}] \label{Proposition:CographP4}
\begin{enumerate} [label=(\arabic*)]
    \item A cograph is a graph that does not contain $P_4$ as an induced subgraph.
    \item A cograph is a graph that can be generated from the following operations:
    \begin{enumerate} [label=(\roman*), left=0pt, rightmargin=0pt, topsep=0pt, itemsep=0pt]
			\item $K_1$ is a cograph.
			\item If $G_1$ and $G_2$ are cographs, then so is $G_1 + G_2$.
			\item If $G_1$ and $G_2$ are cographs, then so is $G_1 \vee G_2$.
		\end{enumerate}
\end{enumerate}
		
\end{proposition}

\smallskip
Virgile \cite{VV} showed that EVICTION ETERNAL DOMINATING SET is EXPTIME-complete and that the eviction number of cographs can be computed in polynomial time. Clearly, the same holds for the graphs listed in Proposition \ref{Proposition:EvictionBasicFamilies}.

\begin{problem}
Find further classes of graphs for which the eviction number
can be computed in polynomial time.
\end{problem}

\medskip

\noindent\textbf{Acknowledgement\hspace{0.1in}} We acknowledge the support of
the Natural Sciences and Engineering Research Council of Canada (NSERC), PIN 04459, 253271.

Cette recherche a \'{e}t\'{e} financ\'{e}e par le Conseil de
recherches en sciences naturelles et en g\'{e}nie du Canada (CRSNG), PIN 04459, 253271.
%BeginExpansion
\begin{center}
\includegraphics[width=2.5cm]{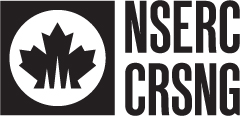}%
\end{center}
%EndExpansion


\begin{thebibliography}{99}                                                                        %


\bibitem {BCG-S}\textsc{A. P. Burger, E. J. Cockayne, W. R. Gr\"{u}ndlingh, C.
M. Mynhardt, J. H. van Vuuren, and W. Winterbach}, Finite order domination in
graphs. \emph{J}. \emph{Combin. Math. Combin. Comput.} \textbf{49} (2004), 159--175.

\bibitem {BCG2}\textsc{A. P. Burger, E. J. Cockayne, W. R. Gr\"{u}ndlingh, C.
M. Mynhardt, J. H. van Vuuren, and W. Winterbach}, Infinite order domination
in graphs. \emph{J}. \emph{Combin. Math. Combin. Comput.} \textbf{50} (2004), 179--194.

\bibitem {CL}\textsc{G. Chartrand, L. Lesniak, and P. Zhang}, \emph{Graphs
\&\ Digraphs }(sixth edition), Chapman and Hall/CRC, Boca Raton, 2016.

\bibitem{corneil1981complement} \textsc{D. G. Corneil, H. Lerchs, L. S. Burlingham}, Complement reducible graphs, \textit{Discrete Appl. Math.} \textbf{3} (1981), no. 3, 163--174.	

\bibitem{corneil1985linear} \textsc{D. G. Corneil, Y. Perl and L. K. Stewart}, A linear recognition algorithm for cographs, \textit{SIAM J. Comput.} \textbf{14} (1985), no. 4, 926--934.

\bibitem {klostermeyer2016dynamic}\textsc{W. F. Klostermeyer, M. Lawrence and
G. MacGillivray}, Dynamic Dominating Sets: the Eviction Model for Eternal
Domination, \emph{J. Combin. Math. Combin. Comput.} \textbf{97} (2016), 247--269.

\bibitem{klostermeyer2007eternal} \textsc{W. F. Klostermeyer and G. MacGillivray}, Eternal security in graphs of fixed independence number, \textit{J. Combin. Math. Combin. Comput.} 63 (2007), 97–101.

\bibitem{klostermeyer2016protecting} \textsc{W. F. Klostermeyer and C. M. Mynhardt}, Protecting a graph with mobile guards, \textit{Appl. Anal. Discrete Math.}, 10 (2016), no. 1, 1–29.
	
\bibitem{klostermeyer2020eternal} \textsc{W. F. Klostermeyer and C. M. Mynhardt}, Eternal and secure domination in graphs, \textit{Topics in domination in graphs, Dev. Math.} 64 (2020), 445–478, Springer, Cham.

\bibitem{ramsey1929logic} \textsc{F. P. Ramsey}, On a Problem of Formal Logic, \textit{Proc. London Math. Soc.} (2) 30 (1929), no. 4, 264–286.

\bibitem {West}\textsc{D. B. West}, \emph{Introduction to Graph Theory},
Prentice-Hall, 1996.

\bibitem {VV} \textsc{V. Virgile}, \emph{Mobile Guards’ Strategies for Graph Surveillance 
and Protection}, Doctoral Dissertation, University of Victoria, 2024. 
\small{https://dspace.library.uvic.ca/items/4ff56e60-1878-4697-85ba-b34bf44ecea3}

\end{thebibliography}
\end{document}